\documentclass{article}
\usepackage{marvosym}
\usepackage{amsmath,amsfonts,amsthm,amssymb}
\usepackage{fancybox}
\usepackage{graphicx}
\usepackage{pdfsync}

\newcommand{\stsets}[1]{\mathbb{#1}}

\newcommand{\F}{\stsets{F}}
\newcommand{\N}{\stsets{N}}

\newcommand{\R}{\stsets{R}}

\newenvironment{Abs}{\vspace*{0.5cm}\begin{center} {\bf
    Abstract}\end{center}\vspace*{0cm}\nopagebreak
   \small  } {\vspace{0.5cm}}

\newenvironment{Key}{\vspace*{0.5cm}\noindent {\bf
    Keywords:}
   \sc } {\vspace{0.5cm}}

\newenvironment{AMS}{\vspace*{0.5cm}\noindent {\bf
    AMS 2010 Subject Classification:} } {\vspace{0.5cm}}

\newenvironment{Ack}{\vspace*{1.5cm} \begin{center}{\bf
    Acknowledgements}\end{center}\vspace*{0cm}\nopagebreak
   \rm } {\vspace{0.5cm}}

\theoremstyle{definition}

\theoremstyle{remark}
\newtheorem{Rem}{Remark}
\newtheorem{Exam}{Example}

\newtheoremstyle{mytheorem}{0.5cm}{0.2cm}{\slshape}{ }{\bfseries}{.}{ }{}
\theoremstyle{mytheorem}
\newtheorem{Th}{Theorem}

\newtheorem{Lem}{Lemma}

\newtheorem{Cor}{Corollary}

\renewcommand{\P}{{\bf P}}
\DeclareMathOperator{\E}{{\bf E}}

\DeclareMathOperator{\one}{{ 1\hspace*{-0.55ex}I}}

\newcommand{\cdf}{\textrm{c.d.f.\ }}

\newcommand{\deq}{\stackrel{{\mathcal{D}}}{=}}
\newcommand{\eps}{\varepsilon}
\renewcommand{\epsilon}{\varepsilon}
\renewcommand{\phi}{\varphi}

\newcommand{\thru}{,\dotsc,}
\newcommand{\seg}{{see, e.~g.,\nolinebreak}}
\newcommand{\iid}{{i.\,i.\,d.\ }}

\newcommand{\salg}{\mbox{$\sigma$}-algebra }

\newlength{\querylen}
\newcommand{\ie}{\textrm{i.\,e.\ }}

\newcommand{\as}{\textrm{a.\,s.\ }}

\newcommand{\bbb}{\mathcal{B}}

\newcommand{\fff}{\mathcal{F}}

\renewcommand{\lll}{\mathcal{L}}
\newcommand{\mmm}{\mathcal{M}}
\newcommand{\nnn}{\mathcal{N}}

\newcommand{\tM}{\tilde{M}}

\newcommand{\heta}{\hat{\eta}}

\newcommand{\beq}{\begin{equation}}
\newcommand{\eeq}{\end{equation}}

\newcommand{\card}{\mathrm{card}}
\newcommand{\ed}{\stackrel{D}{=}}

\begin{document}
\title{LISA: Locally Interacting Sequential Adsorption}
\date{\today}
\author{
Anton Muratov\thanks{Chalmers University of Technology,
    Department of Mathematical Sciences, Gothenburg, Sweden. Email: \texttt{[muratov|sergei.zuyev]@chalmers.se}} \and 
\addtocounter{footnote}{-1}
Sergei Zuyev\footnotemark}
\maketitle

\begin{Abs}
  We study a class of dynamically constructed point processes in which
  at every step a new point (particle) is added to the current
  configuration with a distribution depending on the local structure
  around a uniformly chosen particle. This class covers, in
  particular, generalised Polya urn scheme, Dubbins--Freedman random
  measures and cooperative sequential adsorption models studied
  previously. Specifically, we address models where the distribution
  of a newly added particle is determined by the distance to the
  closest particle from the chosen one. We address boundedness of the
  processes and convergence properties of the corresponding sample
  measure. We show that in general the limiting measure is random when
  exists and that this is the case for a wide class of almost surely
  bounded processes.
\end{Abs}

\begin{Key}
  sequential adsorption, stopping set, point process, random measure,
  Polya urn, convergence of empirical measures
\end{Key}

\begin{AMS}
  primary 60G55;
  secondary 60G57, 60D05, 60F99, 82C22
\end{AMS}
\section{Introduction}
\label{sec:introduction}

A model of sequentially constructed point process that inspired this
paper was presented to one of the authors (SZ) by Richard W.~R.~Darling
as a way to describe a certain population dynamics. His original model
is described as follows. Start with a fixed finite configuration
$X=\{x_1\thru x_{n_0}\}$ of $n_0$ points in a plane. Call them
\emph{particles}. Choose one of these particles uniformly at
random. This particle, say $\xi$, is thought of as a `parent' of a new
particle that will be added to the current configuration according to
the following rule. Consider $k$ closest to $\xi$ particles
$x_1(\xi)\thru x_k(\xi)$, where $k\geq 3$ is a parameter of the model,
and fit a 2-variate Normal distribution centred in $\xi$ to these. Let
$\widehat{V}$ be the corresponding estimate of the covariance
matrix. Then sample a new particle from this estimated law:
$x_{n_0+1}\sim \mathcal{N}(\xi,\widehat{V})$. Once this is done, we
have a configuration of $n_0+1$ particles and we repeat the procedure
again: choose randomly a particle among all $n_0+1$ particles now
present, estimate the Normal law from the closest to it $k$ particles
and add a new particle sampled from this law, etc.

A realisation of the model based on 20 initial particles after 10
thousand steps is shown in the upper-left plot in Figure~\ref{fig:40-thousand-particles}.

\begin{figure}[ht]
  \centering
  \includegraphics[height=13cm]{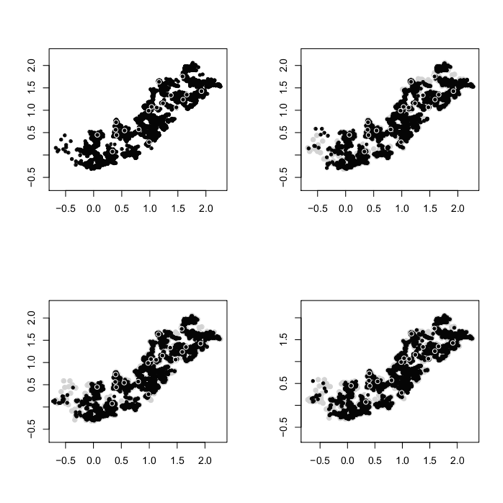}
  \caption{Sequence of 10, 20, 30 and 40 thousand generated
    particles. Newly added ten thousand particles are in dark, previously
    existing -- in grey (partially covered), initial particles are
    contoured void circles.}
  \label{fig:40-thousand-particles}
\end{figure}

One can note the following characteristic features of the
construction. Since the parent particle is chosen uniformly, there is a
greater chance that this parent will be chosen in the area densely
populated by the particles. Moreover, in these dense areas the
distance between the particles tends to be small, so the newly added
particle also tends to lie close to the parent point. So as the
construction progresses, it tends to reinforce the dense
areas of particles which kind of `adsorb' new particles. This is clearly seen in
Figure~\ref{fig:40-thousand-particles}, where the configuration is
shown after 10, 20, 30 and 40 thousand steps. Each of newly added 10
thousand particles are shown in dark emphasising their trend to follow
higher density areas of the previously existing (grey) points. Note that although the
configuration of existing particles plays a crucial role in the construction, only $k$
closest particles to the chosen one actually contribute to the
distribution of the added point. In this sense the interaction is
local, hence the name we have chosen for this process: Locally
Interacting Sequential Adsorption or \emph{LISA}, for short.

Another feature concerns the geometry of the cloud of particles. When
the parent particle is chosen inside a circular cloud, its closest
neighbours tend to be homogeneously spread around it. This produces
more or less isotropic Normal density for a new point wich adds to a
round cloud making it even more isotropic. In contrast, when a
boundary particle is chosen as a parent or when it lies in a stretched
cloud, the density will also be skewed in the corresponding
direction. So in the long run round clouds tend to stay round, but
time to time `shootouts' from their boundary happen which then tend to
produce filamentary arrangements. It also happens due to randomness
that even if a parent in such a filament is chosen, it can still produce a particle well outside
the main direction, and this would then become a centre of another
circular cloud.

There is a range of questions arising immediately: will the particles
be always confined to a bounded region or will the
diameter of the cloud will increase indefinitely? Will eventually
particles be present in any compact set of a positive area or will
there be gaps never filled by the process?
If we supply all $n$ the particles present at the current step with
masses $1/n$ we get a probability sample measure $\nu_n$. Is there a
limit in appropriate sense of the sequence of these measures? Is this
limit measure when exists random or is it non-random? And what about finer
properties of this limiting measure, like the Hausdorff dimension of its
support? 

Surely, the two-variate Normal distribution is just one of possible choices
of the distribution governing addition of a new point. And all the
above questions can be asked for any other distribution: for instance,
to provoke shootouts one would try some heavy-tailed distribution for
the distance from its centre. We, however, want to keep the main
essence of the local interaction of the model above requiring that the
new particle distribution scales appropriately when the configuration
becomes denser. This will bring us to the notion of a stopping set
described in details in the next section.

The structure of the paper is the following. In the next
Section~\ref{sec:gener-model-descr} we fix the notation used
throughout and give formal description of the class of locally
interacting sequential processes we are dealing with, the Darling's
model being one particular case of these. Other cases include such
seemingly different models as Dirichlet measures, Dubbins-Freedman's
random distribution functions and cooperative sequential
adsorption. Section~\ref{sec:rand-limit-distr} demonstrates on a
simple example that the limiting distribution of particles, if exists,
is generally a random measure, this particular example leads to
Dubbins-Freedman construction of a random distribution
function. Section~\ref{sec:bound-sequ-proc} addresses boundedness
issue and show that under rather mild conditions the cloud of points
has almost surely finite diameter. Finally,
Section~\ref{sec:lim-measure} studies convergence of sample measures
and shows that in models with an \as finite diameter such a limiting
measure exists in a weak sense almost surely. LISA processes constitute a very large class of models
with different properties, so we conclude by outlining extensions,
relations to other models and open problems which are abound.

\section{Preliminaries and Model Description}
\label{sec:gener-model-descr}

In order to define a locally interacting sequential adsorption process, we need
a few components.  First, the phase space $W$, where the particles live,
and an initial configuration $X_{n_0}=\{x_1\thru x_{n_0}\}$ of
particles in it which is a parameter of the model. Although a
generalisation is immediate, we assume in this paper that $W$ is a subset of
Euclidean space $\R^d$. It is often
convenient to treat a particle as a unit mass measure so that a
collection of particles is a counting measure on the Borel subsets of $W$.

As already alluded in Introduction, the local interaction, thought of
as a dependence of the distribution of the newly added particle on the local
configuration of particles around its parent, can be described in terms
of a stopping set which is the next component to be defined now.

Let $\mmm$ denote a set of Radon measures on the Borel sets $\bbb$ of
$\R^d$ and $\nnn\subset\mmm$ be the set of counting $\sigma$-finite
measures on $\bbb$. For a closed set $G\in\bbb$, let $\fff_G$ be the
\salg of subsets of $\mmm$ generated by the sets $\{\mu\in\mmm:\
\mu(B\cap G)\leq t\},\ B\in\bbb,\ t\geq0$ and let $\fff=\vee_B
\fff_G$, where $G$ runs through any countable system of bounded closed
sets generating $\bbb$. The system $\{\fff_G\}$ is a \emph{filtration},
because it possesses the following properties:
\begin{enumerate}
\item \emph{Monotonicity:} $\fff_{G'}\subseteq \fff_G$ whenever
  $G'\subseteq G$;
\item \emph{Continuity from above:} $\fff_G=\cap_{n} \fff_{G_n}$ for
  any sequence of closed nested sets: $G_1\supseteq G_2\supseteq\dots$
  such that $\cap_nG_n=G$.
\end{enumerate}
A \emph{random measure} (resp., a \emph{point process}) is a
measurable mapping from some probability space
to $[\mmm,\fff]$ (resp., to $[\nnn,\fff]$). A realisation of a point
process is called a configuration (of particles).

Denote by $\F$ the ensemble of all closed sets of $\R^d$ and by $\Xi$ the smallest
$\sigma$-algebra containing the sets $\{G\in \F:\ G\cap
K\neq\emptyset\}$ for all compact sets $K$. A \emph{random closed set} is a
measurable mapping from a probability space to $[\F,\Xi]$. We will be
working with the canonical space for the point processes when dealing
with random sets so they become a measurable functions of point configurations.

A \emph{stopping set} is a random closed set $S:\
[\nnn,\{\fff_G\},\P]\mapsto [\F,\Xi]$ such that the event
$\{S\subset G\}$ is $\fff_G$ measurable for any $G\in\F$. The
corresponding \emph{stopping \salg}$\fff_S$ consists of events
$E\in\fff$ such that $E\cap\{S\subset G\} \in \fff_G$ for any
$G\in\F$. A stopping set is a generalisation of the classical notion of
a stopping (or Markov) time: likewise a random process' trajectories
stopped at the Markov time, the geometry of a stopping set is
determined by the configuration of particle inside it and on its
boundary and does not depend on the particles outside of the stopping
set.

For more details on stopping sets in $\R^d$, see \cite{Zuy:99b}
and Appendix in~\cite{BauLast:09} covering also more general phase spaces.

Returning to the construction of LISA, secondly, for any point $x\in W$ and
all finite configurations $X$ with $n\geq n_0$ points there is defined
a stopping set $S_x(X\setminus\{x\})$ (by definition,
$X\setminus\{x\}=X$ if $x\not\in X$). In other words, if $X'$ is
another configuration such that $X'\cap S_x(X\setminus\{x\})=X\cap
S_x(X\setminus\{x\})$, then necessarily
$S_x(X'\setminus\{x\})=S_x(X\setminus\{x\})$. From now on, to ease
the notation, we will simply write $S(x,X)$ of just $S_x$ when no
confusion occurs instead of $S_x(X\setminus\{x\})$.

Finally, for every stopping set $S_x$ with the corresponding stopping
$\sigma$-algebra $\fff_{S_x}$ there is defined a random variable
$\zeta_{S_x}$, such that its distribution is defined only by the
geometry of the stopping set $S_x$ and the particles it contains. In
other words, $S_x$ can be viewed as a parameter of this distribution,
or if there are other natural parameters of this distribution, they
are necessarily $\fff_{S_x}$-measurable. Typically, for our purposes
$S_x$ and
$\zeta_{S_x}$ are defined to be shift invariant and scale homogeneous,
so that
\begin{align}
  \label{eq:hom}
  &S(x,X)\deq  x+S(0,X-x) & \zeta_{S(x,X)}&\deq  x+\zeta_{S(0,X-x)}\\
  & S(0,aX)\deq a\,S(0,X) & \zeta_{S(0,aX)}&\deq a\,\zeta_{S(0,X)}
\end{align}
for any positive $a$, configuration $X$ and $x\in W$ (\,$\deq$ denotes
equality in distribution).  In R.~Darling's model described in the
previous section, the stopping set $S(x,X)$ is the smallest closed
ball centred at $x\in X$ containing $k$ nearest neighbour particles of
$X\setminus\{x\}$ to $x$.  The covariance matrix $\widehat{V}$ estimated from
these particles (with or without $x$ itself) defines $\zeta_{S_x}$ as
having Multivariate Normal distribution $\mathrm{MVN}(x,\widehat{V})$ centred at
$x$. Since only the particles contained in $S_x$ are used to estimate
$V$, $\widehat{V}$ is $\fff_{S_x}$-measurable.

Having these necessary components at hand, we define a dynamical
procedure by which new particles are sequentially added to the
existing configuration one by one.  Let $\{\chi_n\}$ be a sequence of
independent random variables, where $\chi_n$ is uniformly distributed
on the discrete set $\{1,2\thru n\}$. Given current configuration
$X_n=\{x_1\thru x_n\}$ of $n\geq n_0$ particles, a new particle
$x_{n+1}$ distributed as $\zeta_{S(x_{\chi_n},X_n)}$ is added to the
configuration. In other words, a particle of $X_n$ is uniformly chosen
(so it is a particle with index $\chi_n$), and then a new particle is
added according to the distribution defined by its stopping set.

We now give examples of models which are constructed this way.

\begin{Exam}\label{ex:rnu}
  Let $W=[0,1],\ n_0=1$ and $X_{n_0}=\{0\}$. The stopping set $S_x$ is
  the segment from $x$ to the next particle to the right (or to 1 if
  there is no such particle). Formally, $S_x=S(x,X)=[x,x+d^+(x,X)]$,
  where $d^+(x,X)=\min \{x'-x:\ x'\in X\setminus\{x\}\cup\{1\},\
  x'>x\}$. Finally, $\zeta_{S_x}$ is a uniformly distributed
  point on $S_x$.
\end{Exam}

In the example above all the added particles belong to $W=[0,1]$ by
construction. In the next model the particles' range grows
indefinitely, but as we show in the next section, all the
particles will be confined to an almost sure bounded (but random) set.

\begin{Exam}\label{ex:clu}
  Let $W=\R$ and $X_{n_0}$ be some set of $n_0\geq 2$ particles. The
  stopping set $S_x=[x-d(x,X),x+d(x,X)]$, where
  $d(x,X)=\min\{|x-x'|\,:\ x'\in X\setminus\{x\}\}$ is the
  distance to the closest to $x$ particle of $X$. As in the previous model,
  $\zeta_{S_x}$ is uniformly distributed in $S_x$.
\end{Exam}

\begin{Exam}\label{ex:cln}
  This is one-dimensional variant of the model described in
  Introduction. Here $W,\ X_{n_0}$, $S_x$ and $d(x,X)$ are as in the
  previous example. But  $\zeta_{S_x}$ is Normally distributed with
  mean $x$ and standard deviation $a d(x,X)$ for some $a>0$.
\end{Exam}

\begin{Exam}\label{ex:cld}
  More generally, let $W=\R^d$, $S(x,X)$ be the closed ball centred in
  $x$ with radius $d(x,X)$ and $\zeta_{S_x}=x+d(x,X)\psi$
  with a given random vector $\psi\in\R^d$ whose distribution does not
  depend on anything. In the previous two
  examples, $d=1$ and $\psi$ is uniformly distributed in $[-1,1]$ in
  Example~\ref{ex:clu} or $\psi$ has normal $\mathcal{N}(0,a^2)$
  distribution in Example~\ref{ex:cln}.
  
\end{Exam}

In the next two examples, the distribution of the $n$th new particle to be
added does not depend on the index variable $\chi_n$, but rather on the whole
current configuration of the particles.

\begin{Exam}
  Let $W$ be some measurable space and $\mu$ be some given probability
  measure on its measurable subsets. Define $S(x,X)$ to be the whole
  $W$ for all $x$ and $X$. Random variable $\zeta_{S_x}$ equals $x$
  with probability $1-1/n$ and otherwise a random variable with
  distribution $\mu$ with probability $1/n$, where $n$ is the
  cardinality of $X$. Surely, the parameter $n$ of the distribution of
  $\zeta_{S_x}$ is $\fff_{S_x}=\fff_W$ measurable. Such defined LISA
  process describes the Blackwell--MacQueen construction which
  generalises the Polya urn scheme. It weakly converges to a Dirichlet
  random measure in the limit, see~\cite{BlaMcQ:73}.
\end{Exam}

\begin{Exam}\label{ex:csa} Let $W$ be some compact subset of $\R^n$
  and $\{\beta_n\}$ be a given sequence of positive numbers. Fix also
  a positive parameter $R$ called the \emph{interaction
    radius}. Define $S_x$ to be a closed ball $B(x,R)$ centred at $x$ with
  radius $R$ and $\zeta_{S(x,X)}$ to be the random variable
  with the density proportional to the function
  $f(x)=f(x,X)=\prod_{k=0}^{|X|} \beta_{n(x,X)}$, where $n(x,X)$ is
  the number of particles from $X$ belonging to $B(x,R)$. The
  corresponding LISA process then defines the so-called
  \emph{cooperative sequential adsorption} (CSA) model,
  see~\cite{Scherb:09,PenrScherb:09} and the references therein.
  
\end{Exam}

When the stopping set $S_x$ is allowed
to be the whole $W$, we are basically in the situation when the
distribution of the added particle depends on the whole current
configuration. Such construction may include just about any
dynamically constructed processes and is too general to be treated in
a unified manner. So to stay in the ``locally interacting'' framework,
we will concentrate in this paper only on LISA processes where the
distribution of $\zeta_{S_x}$ depends only on the distance $d(x,X)$
from $x$ to the (properly defined) closest particle among
$X\setminus\{x\}$, \ie on Examples~\ref{ex:rnu}--\ref{ex:cld}. The
original Darling's model which inspired this investigation does not
fall into this framework (unless $k=1$ in 1D case) and its detailed
analysis is still a hard open problem. But even the models we do
analyse here exhibit fascinating and different behaviours. These
concern, first of all, randomness of the limiting
distribution, boundedness of its support and its dimension.

\section{Random limiting distribution of LISA}
\label{sec:rand-limit-distr}

This section demonstrates that, in general, the limiting distribution
of LISA processes is non-degenerate. We show on Example~\ref{ex:rnu}
that the sample distributions functions $F_n(t)=n^{-1}\sum_{k=1}^n\one_{x_k\leq
  t},\ t\in[0,1]$ of the first $n$ particles converge to a random
distribution function on $[0,1]$ arising in the Dubbins-Freedman
construction, see~\cite{DubFre:66}. This fact has already been noted
in \cite[Sec.~5.2]{Pey:79}, but included here for a didactic purpose.

Recall the Dubbins--Freedman construction of a random measure with support
on $[0,1]$. A realisation of the cumulative distribution function
of such a measure is produced by the following sequential
procedure. Let $u_1=(\phi_1,\psi_1)$ be two independent uniformly
distributed in $[0,1]$ random variables, or, equivalently,
$u_1\sim \mathrm{Unif}([0,1]^2)$. The vertical and the
horizontal lines passing through this point divide the square
$[0,1]^2$ into four rectangles.  The distribution function being
constructed is deemed to pass through the points $(0,0),
(\phi_1,\psi_1)$ and $(1,1)$. Since the \cdf is a non-decreasing
function, it must be contained in the rectangles: $I_{1,1}=[0,\phi_1]\times
[0,\psi_1]$ and $I_{1,2}=[\phi_1,1]\times [\psi_1,1]$ lying along the `main'
diagonal from bottom left to top right. Namely, the \cdf passes
through a uniformly generated point
$u_{1,1}\sim
\mathrm{Unif}(I_{1,1})$ in the first rectangle and through a uniformly
generated point $u_{1,2}\sim \mathrm{Unif}(I_{1,2})$
in the second. Again, both these points divide the
corresponding rectangles $I_{1,1}$ and $I_{1,2}$ into 4 rectangles
each, the diagonal ones containing the \cdf In each 4 of these
diagonal rectangles of level 3 random uniform points are selected which
the \cdf is deemed to pass through, etc. Thus the \cdf is defined on a
everywhere dense set in $[0,1]$ and the values in all other
points are defined as the limits. Thus one obtains a continuous
increasing curve which is a random element on the space of
independent uniform random variables indexed by
a binary tree: $u_1, u_{1,1}, u_{1,2}, u_{1,1,1}, u_{1,1,2},
u_{1,2,1},\dots$.

Consider now the first particle $x_2$ generated in the construction
described in Example~\ref{ex:rnu}. It has $\mathrm{Unif}[0,1]$
distribution. Now an analogy with Polya urns can be drawn in the
following manner. Paint the particle $x_1=0$ black and the particle
$x_2$ white. The next generated particle $x_3$ will be black or white
according to whether it is the black $x_1$ or it is the white $x_2$ selected
on stage 3, \ie $\chi_2=1$ or $\chi_2=2$. Then the procedure repeats
with the colours of the particles being inherited from their `parent'
particles. So the number of black and white particles has the same
distribution as the number of black and white balls in a Polya urn with
starting configuration of one black and one white balls. But all black
particles are lying to the left of $x_2$ and all white are to the right.
So the proportion of the black particles is the proportion of
particles with coordinates less than $x_2$ which equals the proportion $y_2$
of black balls in the urn scheme which is $\mathrm{Beta}(1,1)$, or
equivalently, the uniform distribution on $[0,1]$. Thus the limiting
\cdf passes through the points $(x_2,y_2)$ having the same
distribution as $u_1$ in the Dubbins--Freedman construction.

Conditioning now on the value of the second `daughter' particle $x'$
of $x_1=0$, the proportion of all the particles to the left of it
conforms to $\mathrm{Unif}[0,y_2]$. Indeed, just ignore all the white
particles in the construction and distinguish among all `black'
particles the ones which are really black in $[0,x')$ and `dark grey'
which lie in $[x',x_2)$. So the \cdf passes through the point
distributed as $u_{1,1}$ above. Iterating to other segments, we
conclude the demonstration of the equivalence.

\section{Boundedness of LISA processes}
\label{sec:bound-sequ-proc}

Next natural question to be addressed is whether the limiting
distribution of particles in LISA, when it exists, has an
\as\ bounded support. This is trivially true for
Example~\ref{ex:rnu}, but it is not that evident in other
examples. Notice, that since the series $1/n$ diverges, each particle
will be chosen infinitely many times as a parent point. Thus the
rightmost particle present at stage $n$, for instance, in Example~\ref{ex:clu} will
eventually be chosen and with probability 1/2 will produce a
particle yet more to the right. So the support is growing, but will it
stay compact nevertheless?

\subsection{Boundedness in Example~\ref{ex:clu}.}
Recall Example~\ref{ex:clu}. Let $W=\R$,
$X_{n_0}=X_2=\{0,1\}$. It is convenient to slightly reformulate the rule 
by which new particles are added:
\begin{equation}\label{eq:xin}
  x_{n+1} = x_{\chi_{n}} + \eps_{n+1} d_{n} \eta_{n+1}\,,
\end{equation}
where $\{\eps_n\}$ is a sequence of \iid random variables equal to
$\pm 1$ with probability $1/2$, $\{\eta_n\}$ are \iid uniformly
distributed in $[0,1]$ and given $x=x_{\chi_n}$, $d_n=d(x,X_n) \in
\fff_{S(x,X_n)}$ for the stopping set
$S(x,X_n)=[x-d(x,X_n),x+d(x,X_n)]$.

\begin{Th}\label{prop:bounded1}
  Denote
  \begin{align*}
    m_n &=\min \{x:\ x \in X_n\}\,,\\
    M_n &=\max \{x:\ x \in X_n\}
  \end{align*}
  for the LISA model in Example~\ref{ex:clu}. Then
  \begin{equation}
    \label{eq:fin}
    \P\{-\infty<\liminf m_n\leq \limsup M_n < +\infty\}=1\,.
  \end{equation}
\end{Th}
\begin{proof}
  We only prove that $\limsup M_n$ is finite \as A proof of finiteness
  of $\liminf m_n$ is similar. Introduce $\nu_m$ --- time of the
  $m$-th jump of the process $\{M_n\}$ as follows:
  \begin{align*}
    &\nu_0=n_0\,,\\
    &\nu_{m+1}=\min\{k: M_k>M_{\nu_{m}}\}, m=0,1,2,\dotsc
  \end{align*}
  Now consider imbedded process $\tM_m=M_{\nu_m}$. From
  \eqref{eq:xin} we obtain
  \begin{equation}\label{eq:max1}
    \tM_{m}= x_{\chi_{\nu_{m}-1}} + d_{\nu_{m}-1} \eps_{\nu_{m}}
    \eta_{\nu_{m}}, m=1,2,\dotsc
  \end{equation}
  The distribution of $\eta_n$ is concentrated on $(0,1]$, thus the
  maximum can only have $m$-th jump at time $n$ if
  $\chi_n=\nu_{m}$. That means, on the $(\nu_{m}{-}1)$-th step the
  $(m-1)$-th maximum is chosen, implying that
  $x_{\chi_{\nu_m-1}}=\tM_{m-1}$. Moreover, $\eps_{\nu_m-1}$ must be
  equal to $1$ in order for a positive jump to happen. Thus
  \eqref{eq:max1} is reduced to
  \begin{equation}
    \tM_m = \tM_{m-1} + d_{\nu_m} \eta_{\nu_m},\ \
    m=1,2,3,\dotsc
  \end{equation}
  Notice that $d_{\nu_m}=d_{\nu_m}(\tM_{m-1})$ is less or equal than
  $\tM_{m-1}-\tM_{m-2}$, and for that expression we have
  \begin{displaymath}
    \tM_{m-1}-\tM_{m-2}=d_{\nu_{m-1}}
    \eta_{\nu_{m-1}}, \ \ m=2,3,\dotsc
  \end{displaymath}
  By induction, one can get
  \begin{equation*}
  \tM_{m}-\tM_{m-1} \leq \prod_{k=1}^{m} \eta_{\nu_k},
  \end{equation*}
  so, since $\tM_0 = 1$, we come to a bound
  \begin{equation} \label{est:max}
    \tM_{m}\leq 1+ \sum_{l=1}^{m} \prod_{k=1}^{l} \eta_{\nu_k},
    m=1,2,3,\dotsc
  \end{equation}
  Notice now that $\{\eta_{\nu_k}\}_{k\geq 1}$ and $\{\eta_k\}_{k\geq 1}$
  are equally
  distributed, since by the definition of $\nu_m$,
  $n_0=\nu_1<\nu_2<\nu_3<\dotsc,$ and $\{\eta_k\}$ are independent of
  $\{\nu_m\}$. The sequence $\{\tM_{m}\}$ is monotonely increasing,
  therefore it has an \as\ limit $\tM_\infty$, although possibly
  infinite. However, since $\E\eta_n<1$,
  \begin{displaymath}
    \E\lim_{m\to\infty}\tM_m = \lim_{m\to\infty}\E\tM_m < \infty,
  \end{displaymath}
 in particular, $\tM_\infty<\infty$ \as\ thus
  finishing the proof.
\end{proof}

\begin{Rem}
  Same proof with minor tweaks works for any initial configuration
  $X_{n_0} = \{x_1, x_2, \ldots, x_{n_0}\}$ and even for $\eta$ in the
  scheme~\eqref{eq:xin} distributed with an arbitrary law $F$
  concentrated on $(0,1]$ with $\P(\eta_n<1)>0$ (to guarantee a
  convergence of the sum below).
  Estimate~\eqref{est:max} turns into the following:
  \begin{displaymath}
    \tM_{m}\leq d_{n_0}(\tM_0)\Big(\tM_0 + \sum_{l=1}^{m}
    \prod_{k=1}^{l} \eta_{\nu_k}\Big),
    m=1,2,3,\dotsc
  \end{displaymath}
  Here $\tM_0$ and $d_{n_0}(\tM_0)$ are constants dependent on
  $X_{n_0}$ only.
\end{Rem}

\begin{Rem}
  Dubbins--Freedman random distribution functions $F(t)$ arising in
  Example~\ref{ex:rnu} are almost surely continuous, but also each
  point $t\in [0,1]$ is almost surely a point of growth of the
  distribution function $F(t)$, \ie for any $\epsilon>0$ there exist
  $t'\in (t-\epsilon,t)$ and $t''\in (t,t+\epsilon)$ such that
  $F(t')<F(t)<F(t'')$. In contrast, Example~\ref{ex:clu} provides
  random distribution functions which are continuous, but also contain
  constant regions, \ie the corresponding limiting measure does not
  have connected support. To see this, observe that with positive
  probability there happen to be a configuration of points generated
  by the algorithm of Example~\ref{ex:clu} where there are 2 pairs of
  points: $x_1< x_2< y_1< y_2$ each pair consisting of closely
  situated points separated by a relatively large void, \ie
  $x_2-x_1, y_2-y_1$ are small but $(y_1-x_2)/(x_2-x_1)$ and
  $(y_1-x_2)/(y_2-y_1)$ are large. The construction of new points
  scales with distance, so that the evolution of the initially present
  pair of points at distance $\delta$ has the same distribution as the
  evolution of two initial points at the distance 1 scaled by factor
  $\delta$.  Therefore according to just proven
  Proposition~\ref{prop:bounded1} with a positive probability the
  maximum of all the offsprings of the pair $x_1, x_2$ (affected only
  by the distance $x_2-x_2$) will be strictly smaller than the minimum
  of all the offspring of the pair $y_1, y_2$ (based only on
  $y_2-y_1$) so that there will be a void somewhere between $x_2$ and
  $y_1$ not filled with any points.
\end{Rem}

\subsection{Boundedness in Example~\ref{ex:cld}.}
Important feature of the model considered in the previous section is
that the particles cannot jump over each other and thus the influence
of a new added particles can be effectively controlled. This is now
longer the case in Example~\ref{ex:cld} where the farthest particle
can be potentially generated by any parent point. In this subsection
we present sufficient conditions, under which the more general
$d$-dimensional LISA process from Example~\ref{ex:cld} is bounded \as

Let $W=\R^d$,\ $d\in\N$. Initial configuration is given by
$X_0=\{x_1,x_2,\dotsc,x_{n_0}\}$.  New particles are added
according to the rule:
\begin{equation}\label{eq:xin}
  x_{n+1} = x_{\chi_{n}} + d_{n}(x_{\chi_{n}})\, \psi_{n+1},\ n=n_0, n_0+1, \dotsc
\end{equation} 
As before, $\chi_n$ are independent, distributed uniformly on
$\{0,1,\dotsc,n\}$, $d_n(x)$ is the distance from $x$ to
$X_{n}{\setminus}\{x\}$, $\psi_n$ are \iid random variables with a
given distribution which may now have a non-compact support. Next, we
set $\eta_n=\|\psi_n\|$ and denote $C=\E\eta_1$.  Also, put $\heta_n =
\eta_n {\wedge} 1 = \min\{\eta_n, 1\}$ with corresponding
$\hat{C}=\E\heta_1$.

\begin{Lem}\label{lem:markov}
  Let $\{\eta_n\}$, $\{\phi_n\}$ be \iid sequences of non-negative random
  variables, independent between themselves. Put $\heta_n=\eta_n\wedge 1$,
  $\theta_n = \eta_n \prod_{i=1}^{n-1}\heta_i$. Let also $C=\E\eta_1$ and
  $\E\phi_1$ be finite. Put $\hat{C}=\E\heta_1$. Assume $\hat{C}<1$. Define
  $$
    Y_n= \bigvee_{i=1}^{n}\theta_i(\phi_i+1)
  $$
  Then $Y_{\infty} = \max\limits_n Y_n <\infty$ \as Moreover,
  \begin{equation}\label{eq:est}
    \E Y_{\infty} \leq \frac{C}{1-\hat{C}}(1+\E\phi_1)
  \end{equation}
\end{Lem}

\begin{proof}
  Note that $Y_n$ are monotone and thus converge to some (possibly infinite)
  limit. However, one can show that $Y_n$ are bounded in $\lll^1$:
  $$
    \E Y_n= \E \bigvee_{i=1}^n \theta_i(\phi_i+1) \leq \sum_{i=1}^n \E
    \theta_i (\phi_i+1) = \sum_{i=1}^n \E\eta_i\prod_{j=1}^{i-1}
    \E\heta_j (\E\phi_i+1) =
  $$
  $$
    \sum_{i=1}^n C(\E\phi_i+1)\hat{C}^{i-1} = C(\E\phi_1+1)
    \sum_{k=0}^{n-1}\hat{C}^k<\frac{C\E\phi_1+1}{1-\hat{C}}.
  $$
  Hence $Y_\infty$ has finite expectation with a correspondent bound.
\end{proof}

\begin{Th}\label{th:2}
  If $C+\hat{C}<1$ then $\sup\limits_n|\xi_n|<\infty$ \as Moreover,
  $$
    \E\sup_n|\xi_n| \leq A_0+\frac{n_0 D_0 C}{1-\hat{C}-C}\,,
  $$
 for some constants $A_0, D_0$ depending only on the initial
  configuration $X_{n_0}$.
\end{Th}

\begin{proof}
  First of all, organize $\{x_n\}$ in a tree in the following natural way:
  for every
  point $x_{i_0}$ from $X_{n_0}$ denote all the points $x_n$ such
  that $\chi_{n-1}=i_0$ as $\{x_{i_0i_1}\}_{i_1=1}^\infty$, in the order
  of appearance. We will further say that $x_{i_0i_1}$ are the children of
  $x_{i_0}$. Then all the children of $x_{i_0i_1}$ we
  denote by $\{x_{i_0i_1i_2}\}_{i_2=1}^\infty$, and so on. 

  Let $\nu(i_0\dotsc i_k)$ denote the (random) time of appearance of
  $x_{i_0\dotsc i_k}$. Let also $\psi_{i_0\dotsc i_k} :=
  \psi_{\nu(i_0\dotsc i_k)}$, $\eta_{i_0\dotsc i_k} :=
  \eta_{\nu(i_0\dotsc i_k)}$, $d_{i_0\dotsc i_k}(\cdot) := d_{\nu(i_0\dotsc i_k)-1}(\cdot)$. Observe that
  $\{\psi_{i_0\dotsc i_k}\}_{k,i_0,\dotsc,i_k\in\N}$ and hence 
  $\{\eta_{i_0\dotsc i_k}\}_{k,i_0,\dotsc,i_k\in\N}$ are \iid
  families.

  Fix $i_0\leq n_0$ for now. We have in our new notation:
  \beq\label{est:i0}
    \bigvee_{i_1}|x_{i_0i_1}-x_{i_0}| = \bigvee_{i_1} |x_{i_0}
    + d_{i_0i_1}(x_{i_0})\psi_{i_0i_1} - x_{i_0}| =
    \bigvee_{i_1}|d_{i_0i_1}(x_{i_0})\psi_{i_0i_1}|
  \eeq
  Estimate $d_{i_0i_1}(x_{i_0})$, $i_1=1,2,\dotsc$ Recall that
  it is a distance to the closest neighbour, hence it can not be larger
  than distance to the points that already exist for sure at the moment
  of $x_{i_0i_1}$'s appearance, that is, its mother $x_{i_0}$ and all of
  the older sisters $x_{i_0 1},x_{i_0 2},\dotsc,x_{i_0,i_1-1}$. We can write:
  \begin{align}
    d_{i_0 1}(x_{i_0}) &\leq d_{n_0}(x_{i_0}) =:d_{i_0} \nonumber\\
    d_{i_0 2}(x_{i_0}) &\leq d_{i_0 1}(x_{i_0})\wedge|x_{i_0 1}-x_{i_0}|
    \leq d_{i_0 1}(x_{i_0})\wedge|d_{i_0 1}(x_{i_0})\psi_{i_0 1}| \nonumber\\
    &\leq d_{i_0 1}(x_{i_0})(1\wedge|\psi_{i_0 1}|) \leq d_{i_0}\heta_{i_0 1}\nonumber\\
    d_{i_0 3}(x_{i_0}) &\leq d_{i_0 2}(x_{i_0}) \wedge |x_{i_0 2}-x_{i_0}| \leq
    d_{i_0 2}(x_{i_0}) \wedge |d_{i_0 2}(x_{i_0}) \psi_{i_0 2}|\nonumber\\
    &\leq d_{i_0 2}(x_{i_0})(1\wedge|\psi_{i_0 2}|) \leq d_{i_0}\heta_{i_0 1}\heta_{i_0 2}\nonumber\\
    \vdots &\nonumber\\
    d_{i_0i_1}(x_{i_0}) &\leq d_{i_0}\prod_{j=1}^{i_1-1}\heta_{i_0j}\label{est:d}\\
    \vdots &\nonumber
  \end{align}

  Introduce $\theta_{i_0\dotsc i_k}:=\eta_{i_0\dotsc i_{k-1}i_k}\prod
  \limits_{j=1}^{i_k-1}\heta_{i_0\dotsc i_{k-1} j}$.
  Using~\eqref{est:d} we can estimate the right part of~\eqref{est:i0}:
  \begin{align*}
    \bigvee_{i_1}|d_{i_0i_1}(x_{i_0})\psi_{i_0i_1}| &\leq
    d_{i_0}\bigvee_{i_1}|\psi_{i_0i_1}|\prod_{j=1}^{i_1-1}\heta_{i_0j}
    = d_{i_0}\bigvee_{i_1}\eta_{i_0i_1}\prod_{j=1}^{i_1-1}\heta_{i_0j}\\
    &= d_{i_0}\bigvee_{i_1}\theta_{i_0i_1} =: d_{i_0}L^{i_0}_1
  \end{align*}
  Apply Lemma~\ref{lem:markov} with $\theta_n=\theta_{i_0 n}$, $\phi_n=0$,
  $n=1,2,\dotsc$ to see that $L^{i_0}_1$ is a proper random variable with
  $\E L^{i_0}_1\leq \frac{C}{1-\hat{C}}$.

  Now estimate the second generation.
  \begin{align}
    \bigvee_{i_1,i_2}|x_{i_0i_1i_2}-x_{i_0}| &\leq \bigvee_{i_1}
    \left(|x_{i_0i_1}-x_{i_0}|+ \bigvee_{i_2}|x_{i_0i_1i_2}
    -x_{i_0i_1}|\right)\nonumber\\ &\leq \bigvee_{i_1}\left(
    |x_{i_0i_1}-x_{i_0}| + \bigvee_{i_2}d_{i_0i_1i_2}
    (x_{i_0i_1})\eta_{i_0i_1i_2}\right)\label{est:i2}
  \end{align}
  Note that
  $$
    d_{i_0i_1i_2}(x_{i_0i_1}) \leq d_{i_0i_1 1}(x_{i_0i_1})
    \prod_{j=1}^{i_2-1}\heta_{i_0i_1j}\leq|x_{i_0i_1}-x_{i_0}|
    \prod_{j=1}^{i_2-1}\heta_{i_0i_1j}
  $$
  and therefore we continue~\eqref{est:i2}:
  \begin{align*}
    &\leq\bigvee_{i_1}\left(|x_{i_0i_1}-x_{i_0}|\left(1+\bigvee_{i_2}
    \eta_{i_0i_1i_2}\prod_{j=1}^{i_2-1}\heta_{i_0i_1j}\right)\right)\\ &\leq
    d_{i_0}\bigvee_{i_1}\left(\theta_{i_0i_1}\left(1+\bigvee_{i_2}
    \theta_{i_0i_1i_2}\right)\right)=:d_{i_0}L^{i_0}_2
  \end{align*}
  Note that $\{\bigvee\limits_{i_2}\theta_{i_0i_1i_2}\}_{i_1=1}^\infty$ is
  an i.i.d. sequence, distributed like $L^{i_0}_1$ and independent of
  $\{\theta_{i_0i_1}\}_{i_1=1}^{\infty}$. Using Lemma~\ref{lem:markov}
  again, we obtain $L^{i_0}_2<\infty$ \as, and moreover,
  $$
  \E L^{i_0}_2 \leq \frac{C}{1-\hat{C}} + \left(\frac{C}{1-
  \hat{C}}\right)^2
  $$

  Repeating that argument for $n=3,4,\dotsc$ we obtain a monotone sequence
  $\{L^{i_0}_n\}_{n=1}^\infty$, where $d_{i_0}L^{i_0}_n$ is an \as bound for
  elements from $n$-th generation of descendants of the point $x_{i_0}$,
  $n=1,2,\dotsc$ Since $\{L^{i_0}_n\}$ is monotone and is bounded in $\lll^1$:
  $$
    \E L^{i_0}_n \leq \sum_{i=1}^n\left(\frac{C}{1-\hat{C}}\right)^i
    < \frac{C}{1-C-\hat{C}}
  $$
  there exists an \as limit $L^{i_0}_\infty := \lim\limits_{n\to\infty}L^{i_0}_n$
  with $\E L^{i_0}_\infty \leq \frac{C}{1-C-\hat{C}}$.
  We finish the proof by recalling arbitrariness of $i_0$:
  $$
    \E\sup_{n\geq 1}|x_n| \leq \sup_{i_0\leq n_0}|x_{i_0}|+
    \E\sup_n\sup_{i_0,\dotsc,i_n}|x_{i_0\dotsc i_n}-x_{i_0}|\leq
    A_0 + \frac{n_0 D_0 C}{1-\hat{C}-C}
  $$
  Here $A_0=\sup\limits_{i_0\leq n_0}|x_{i_0}|$, $D_0=\sup\limits_{i_0
  \leq n_0} d_{n_0}(x_{i_0})$
\end{proof}

As an illustration, the 1D Darling's model in Example~\ref{ex:cln} is bounded if
$a<0.8239$ which is a value obtained numerically.

\begin{Rem}\label{rem:conveq}
  The bound in~\eqref{eq:est} and therefore the condition $C+\hat{C} <1$
  may be improved if one could find "nicer" conditions sufficient for
  $$
  \E H = \E \max_{k\geq0} (\eta_k \heta_1\heta_2 \ldots \heta_{k-1}) < 1
  $$
  for \iid non-negative $\eta_k$. We demonstrate that for a particular
  distribution of $\eta_k$. If we introduce $\zeta_k = \log \eta_k$, with
  the distribution function $F(t)$, then the calculation shows that
  $G(t) = \P(\log H < t)$ must satisfy the integral equation:
  \begin{equation}\label{eq:int-eq}
    G(t) = \begin{cases}
      \int_{-\infty}^0 G(t-z)\,dF(z)+(F(t)-F(0))G(t),& t\geq 0\\
      \int_{-\infty}^t G(t-z)\,dF(z),& t<0.
    \end{cases}
  \end{equation}
  Let $\eta_k$ have the following distribution:
  $$
    \P(\eta_k<y)=\begin{cases}
      0, & y<0,\\
      (1-p)y^\beta, & 0\leq y<1,\\
      1-py^{-\alpha}, & y\geq 1.
    \end{cases}
  $$
  for some $\alpha>1,\beta>0,0<p<1$.
  In that case, $\zeta_k=-\log\eta_k$ has \cdf
  $$
    F(t)=\begin{cases}
      (1-p)e^{\beta t},& t\leq 0,\\
      1-pe^{-\alpha t},& t>0,
    \end{cases}
  $$
  and one can directly obtain the solution for~\eqref{eq:int-eq},
  $$
    G(t)=\begin{cases}
      (\frac{p}{1-p}+1)^{-1-\frac{\beta}{\alpha}}e^{\beta 1},&t<0,\\
      (\frac{p}{1-p}e^{-\alpha t+1})^{-1-\frac{\beta}{\alpha}},&t\geq 0
    \end{cases}
  $$
  leading to the following distribution for $H$:
  $$
    \P(H<x)=\begin{cases}
      0, & x<0,\\
      (\frac{p}{1-p}+1)^{-1-\frac{\beta}{\alpha}}x^\beta,&0\leq x<1,\\
      (\frac{p}{1-p}x^{-\alpha}+1)^{-1-\frac{\beta}{\alpha}},&x\geq 1.
    \end{cases}
  $$
  It is not possible to find an analytic form for $\E H$ in that
  generality, but if we fix $\beta$ to be, say, $1/4$, then we can
  find the regions where $\E H<1$ and $\E \eta+\E\heta<1$ numerically:
  \center\includegraphics[scale=0.4]{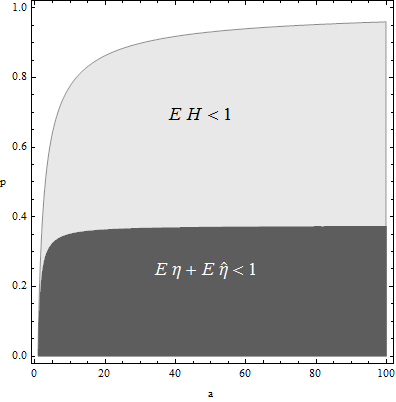}
  
  So, as we see, the condition $\E\eta+\E\heta<1$ of
  Theorem~\ref{th:2} is far from being tight for boundedness.
\end{Rem}

\section{Properties of the limiting measure}
\label{sec:lim-measure}

In previous sections we addressed the boundedness of the series of
point configurations $X_n$. In this section we study the limiting
sample measure and its properties. We are still working with most
general model of Example~\ref{ex:cld} in the phase space
$W=\R^d$,~$d\geq 1$ and initial
configuration $X_{n_0}=\{x_1,\dotsc,x_{n_0}\}$. New particles are added
according to the rule:
\begin{displaymath}
  x_{n+1} = x_{\chi_n}+d_n(x_{\chi_n})\,\psi_n,\ \ n=n{+}1,n{+}2,\dotsc
\end{displaymath}
so that $X_{n+1}=X_n\cup \{x_n\}$. Again, $\chi_n$ are independent,
uniformly distributed over $\{1,2,\dotsc,n\}$, $\psi_n$ are \iid
random variables distributed according to a given probability measure
$\mu$, $d_n(x)$ is, as before, a distance from $x$ to its closest
neighbour in $X_n\setminus\{x\}$. Denote by
$$
  \mu_n=\frac{1}{n}\sum_{k=1}^{n}\mu_{n,k} = \frac{1}{n}\sum_{k=1}^{n}
  \mu(d_n^{-1}(\ \cdot\ -x_k))
$$
the distribution of $x_{n+1}$, and denote by
$$
  \nu_n=\frac{1}{n}\sum_{k=1}^n \delta_{x_k}
$$
the empirical measure of the process $X_n$ after $n$ steps.

We will start with a short lemma providing some insight on the behaviour
of $d_n(x)$.

\begin{Lem}
Assume that $\bigcup\limits_n X_n=X_\infty$ is \as bounded and
$\P(|\psi_1|<1)>0$. Then 
$$\lim\limits_{n\to\infty}\max_{1\leq k\leq n}d_n(x_k)=0$$
\end{Lem}
\begin{proof}
First, notice that for every $k$, $\lim\limits_{n\to\infty}d_n(x_k)=0$.
This follows from $\P(|\psi_1|<1)>0$ and $\sum\P(\chi_n=k)=\infty$,
\ie every point is going to be picked up an infinite number of times
and infinite number of times its $d_n(x_k)$ is going to shrink.

Next, assume the contrary. Let
$$\limsup_n\max_{1\leq k\leq n}d_n(x_k)>\eps>0$$
Pick $\{k_j\},\{n_j\}$ such that
$$
d_{n_j}(x_{k_j})>\eps, j\in\N
$$
Since $d_n(x_k)$ monotonely tends to zero as $n$ goes to infinity,
we can assume all $k_j$ to be different and moreover,
$d_{k_j}(x_{k_j})\geq d_{n_j}(x_{k_j})$. That means, in particular,
that
$$
  |x_{k_i}-x_{k_j}|>\eps, \ \ j\in\N,\ i<j,
$$
\ie $X_\infty$ can't be covered with a finite number of balls of
radius $\eps$ ---
a contradiction with the \as boundedness.
\end{proof}

\begin{Lem}
Assume that $\bigcup\limits_n X_n=X_\infty$ is \as bounded. 
Assume that one of the weak limits $\lim\limits_{n\to\infty}\mu_n$,
$\lim\limits_{n\to\infty}\nu_n$ exists \as and is equal to $\mu^*$. Then
the other one exists \as and is equal to $\mu^*$, too.
\end{Lem}
\begin{proof}
Let $\rho$ be the Levy-Prokhorov distance between probability distributions. We will use
the following property:
\begin{equation}\label{eq:levy-prokh-prop}
  \rho\left(\sum\alpha_k F_k,\sum\alpha_k G_k\right)\leq \max\left(1,
  \sum \alpha_k\right)\max\rho(F_k,G_k)
\end{equation}
Taking it into account, one can write:
\begin{eqnarray*}
  \rho\left(\frac{1}{n}\sum_{k=1}^n\delta(\cdot-x_k),\frac{1}{n}\sum_{k=1}^n
  \mu_{n,k}\right)\leq\max_{1\leq k\leq n}\rho\left(\delta(\cdot-x_k),
  \mu_{n,k}\right)\leq\\
  \rho(\delta(\cdot),\mu(\max_{1\leq k\leq n}
  d_n^{-1}(x_k)\cdot))\to 0,\ \ n\to \infty,
\end{eqnarray*}
since 
$$\lim\limits_{n\to\infty}\max\limits_{1\leq k\leq n}d_n(x_k)=0,$$
and for every $\mu$ -- probability measure on $\R$, $\mu(a_n\cdot)\to\delta$,
whenever $a_n\to \infty$.
\end{proof}

As we have noted, Example~\ref{ex:rnu} is equivalent to the
Dubbins--Freeman construction, so the limiting measure $\mu^*$
exists. The Hausdorff dimension of its support, in a slightly more
general setting, was found in~\cite{KinPit:64}, which is equal to 1/2
here. We are going to use the technique from \cite{Pey:79} to show
that a limiting measure exists in Example~\ref{ex:clu}, however, it is
still an open question, how to calculate the Hausdorff dimension of its
support and whether a limit exists at all in Example~\ref{ex:cln}.

Now, we have $W=\R$, $\psi_n\sim \mathrm{Unif}(-1,1)$. Let $\{x_{n,k}\}_{k=1}^n,
n=1,2,\dotsc$ denote the rearrangement of the elements of $X_n$ in
ascending order. Then the complement of $X_n$ consists of $n{+}1$
intervals, $I_{n,j}$:
\begin{equation*}
\begin{split}
I_{n,0}&=({-}\infty,x_{n,1})\\
I_{n,k}&=[x_{n,k},x_{n,k{+}1}),k=1,2,\dotsc,n{-}1,\\
I_{n,n}&=[x_{n,n},{+}\infty)
\end{split}
\end{equation*}

\begin{Lem}[cf.Lemma 2.1 in \cite{Pey:79}]\label{lem:peyrr1}
  Let $0\leq b\leq a$ be integers, $F\subset (0,1,2,\dotsc,a)$,
  $\card F=b$. Then
  $$
    \mu_n(\bigcup_{j\in F}I_{a,j})\to z_{(a,F)}
  $$
  Here $z_{(a,F)}$ is independent of $\{\chi_j\}_{1\leq j\leq a}$, and
  its law has density
  $$
    \frac{\Gamma(a)}{\Gamma(\hat{b})\Gamma(a-\hat{b})}t^{\hat{b}-1}
    (1-t)^{a-\hat{b}-1}
  $$
  for $\hat{b}=b-\frac{1}{2}(\one(0\in F)+\one(a\in F))$.
\end{Lem}
\begin{proof}
  Follows from the result on the generalised Polya's urn scheme:
  probability for a new
  point to appear in $I_{n,j}$ is $\frac{1}{2n}$ if it is $I_{n,0}$ or
  $I_{n,n}$ and $\frac{1}{n}$ otherwise.
\end{proof}

\begin{Cor}
  $$
    \lim_{n\to \infty}\sup_{1\leq j\leq n}\lim_{m\to\infty}\mu_m(I_{n,j})=0
  $$
\end{Cor}

\begin{Th}
  Almost surely exists $\mu^*$ -- a probability measure such that\\
  $\mu_n\stackrel{w}{\to} \mu^*$.
\end{Th}

\begin{proof}
  By Lemma~\ref{lem:peyrr1} we can define almost surely an increasing
  function $G(x)$ on $X_{\infty}$: $G(x):=\lim\limits_{m\to\infty}
  \mu_m(({-}\infty,x))$. Note that since almost surely
  $$
    {-}\infty<\inf X_\infty<\sup X_\infty < {+}\infty,
  $$
  thus $\sup(x:G(x)=0)>{-}\infty$ and $\inf(x:G(x)=1)<{+}\infty$.
  Moreover, by corollary,
  $$
    \sup(G(x):x<y) = \inf (G(x): x>y),
  $$
  therefore, almost surely $G(x)$ is a continuous
  cumulative distribution function for some probability measure $\mu^*$,
  which is easily shown to be a weak limit of $\{\mu_n\}$.
\end{proof}

We will now prove a couple of facts about the third model. First, let
us make the following observation. Denote by $d^*_n$ the maximal spacing
of a configuration $X_n$:
$$
  d^*_n=\max\limits_{1\leq i\leq n}d_n(x_i)
$$

\begin{Th}
  If $\nu_n=\sum\limits_{i=1}^{n}\delta_{x_i}$ is an empirical measure
  of the process $X_n$ on the $n$-th step, then the Levy-Prokhorov
  distance between the
  empirical measures for the two consecutive configuration is given
  by the following expression
  $$
    \rho(\nu_n,\nu_{n+1})=\min(n^{-1},\max(n^{-2},d^*_n))\,.
  $$
\end{Th}

We will now present some bound for the decrease of the maximal spacing. The
technique we use is essentially due to~\cite{Pey:79}.

Introduce $\phi(t)=1-(\E\eta^t+\E\heta^t)$ and $\sigma=\sup\limits_{
t>0,\phi(t)<1}\{\frac{\phi(t)}{t}\}$. 

\begin{Th}
  If $C+\hat{C}<1$, then
  $\limsup\limits_{n\to\infty}n^\sigma d^*_n<\infty\ \as$
\end{Th}

\begin{proof}
  First step of the proof is to estimate $d_n(x_k)$ from above with a
  certain well-behaving construction. Introduce a triangular array
  $\{\Delta_{n,k}\}\ 1\leq k\leq n,\ n\geq 1$ as follows.
  $$
    \Delta_{n_0,k}=d_{n_0}(x_k),k=1,\dotsc,n_0,
  $$
  For $n\geq n_0$ we put
  $$
    \Delta_{n+1,k}=\begin{cases}
      \Delta_{n,k}\heta_n, & \text{if\ } k=\chi_n,\\
      \Delta_{n,k}\eta_n, & \text{if\ } k=n+1,\\
      \Delta_{n,k} & \text{otherwise.}
    \end{cases}
  $$
  In other words, at each step pick one "diameter" and replace it by
  two diameters, scaled with the realisations of $\eta$ and $\heta$.
  That corresponds to the new point being added to the configuration
  at the $n+1$-th step with its initial distance, and it's mother point's
  distance being scaled correspondently.
  
  It is not very hard to see that for any $n\geq n_0,k\leq n$, one has
  a bound
  \begin{equation}\label{eq:maxspacebnd}
    d_n(x_k)\leq \Delta_{n,k}.
  \end{equation}
  Now, prove that the configuration $\{\Delta_{n,k}\}$ behaves nicely.
  Pick a positive $t$ so that $\E\eta^t+\E\heta^t<1$. Such $t$
  exists, because of the starting conditions.
  Consider the quantity $\sum_{k\leq n}\Delta_{n,k}^t$.
  $$
    \sum_{k\leq n+1}\Delta_{n+1,k}^t=\sum_{k\leq n}\Delta_{k,n}^t-
    \Delta_{n,\chi_n}^t+\Delta_{n,\chi_n}^t(\eta_n^t+\heta_n^t)=
    \sum_{k\leq n}\Delta_{n,k}^t-\Delta_{n,\chi_n}^t\phi(t)
  $$
  Then if $\fff_n$ is the sigma-algebra generated by the evolution of
  $X_n$ up to time $n$, one has
  $$
    \E\Big(\sum_{k\leq n+1}\Delta_{n+1,k}^t\vert\fff_n\Big)=\Big(1-\frac{\phi(t)}{n}\Big)
    \sum_{k\leq n}\Delta_{n,k}^t,
  $$
  and so ${\prod_{1\leq j<n}
  (1-\frac{\phi(t)}{j})^{-1}}{\sum_{k\leq n}\Delta_{n,k}^t}$, together with 
  sigma-algebra $\fff_n$, is a positive martingale, which has an \as
  finite limit. However,
  $$
    n^{\phi(t)}\prod_{1\leq j<n}(1-\frac{\phi(t)}{j})=\prod_{1\leq j<n}
    (1+\frac{1}{j})^{\phi(t)}(1-\frac{\phi(t)}{j})
  $$
  has a positive limit as well, because $0<\phi(t)<1$. Therefore
  $$
    n^{\phi(t)}\sum_{k\leq n} \Delta_{n,k}^t
  $$
  converges to a finite limit as $n\to\infty$. That, together
  with~\eqref{eq:maxspacebnd}, finishes the proof.
\end{proof}

We conclude our paper with noting that all of our models can imbedded
into the continuous time in a natural way: each particle 
produces children independently according to a Poisson process with
intensity 1, and new points' distribution depends only on the local
configuration around the mother point at the time of birth.
The embedded processes of particle
births is equivalent to the original LISA process. Indeed, when $n$
particles are present at any given time $t$, because of the lack of
memory of the exponential distribution, the next one to give birth is
uniformly distributed among them. The model in Example~\ref{ex:rnu}
then becomes the so-called fragmentation or stick breaking process,
\seg \cite{ber:06} and the references therein. In continuous time
version of LISA, the total number of particles is a continuous time
Galton-Watson process of pure birth. In that case, one can also obtain
an asymptotic bound for the maximal spacing of the point
configuration.

\begin{Th}
  If $X_t$ is the continuous time version of the third model, $d^*_t$
  is the maximal spacing of $X_t$, then
  $$
    \limsup_{t\to\infty}e^{-(1-\hat{C}-C)t}\E d^*_t <\infty.
  $$
\end{Th}
The proof is essentially representing $\{X_t\}_{t\in[0,T]}$ as a
trajectory-wise limit as $n$ goes to infinity of a discrete-time process
$\{\tilde{X}{(m)}_l\}_{1<l<\lfloor mT\rfloor}$ in which each particle
at each step gives birth to a new one with probability $p_m=\frac{1}{m}$,
and then observing that if we implement $\{\tilde{\Delta}_{l,k}\}^{1\leq k\leq
l}_{1\leq l\leq \lfloor mT\rfloor}$ similarly to the previous proof, so
that $\tilde{\Delta}_{l,k}$ is a bound for $\tilde{d}_l(\tilde x_k)$, then
one can write out the following distributional inequality
$$
  \tilde \Delta^*_{l+1}\ed(1-I_l)\tilde \Delta_l^*+I_l(\tilde \Delta^*_l\heta_l\wedge
   \tilde \Delta^{*1}_l\eta_l)\leq
   (1-I_l)\tilde \Delta_l^*+I_l(\tilde \Delta^*_l\heta_l+
   \tilde \Delta^{*1}_l\eta_l)
$$
where $\tilde \Delta^{*1}_l$ is an independent copy of $\tilde \Delta^*_l$.
Therefore, one can argue that
$$
  \E \tilde \Delta^*_{\lfloor mT\rfloor}\leq (1-(1-C-\hat{C})p_m)^{
  \lfloor mT\rfloor}
$$
which in the limit brings us to
$$
  \E \tilde \Delta^*_t\leq e^{-(1-\hat{C}-C)T}\,.
$$

\section{Conclusion and open problems}
\label{sec:concl-open-probl}

We have defined a wide class of dynamically constructed point
processes where new particles are added randomly with their
distribution depending on a local neighbourhood of
a randomly uniformly selected particle. Exact notion of locality is
based on stopping sets methodology. In this paper we considered only
the case of models where this set is the ball centred in the selected
particle with the radius equal the distance to the closest
particle. Obvious generalisation is to consider the ball to the $k$-th
nearest neighbour, as it is the case in the original Darling's model
described in Introduction. But even for the case $k=2$ we were not
able to find ways to control the spread of the particles' cloud. So
all the questions of boundedness and/or existence of a limiting
measures remain largely open. 

Boundedness results we established in
Section~\ref{sec:bound-sequ-proc} are only sufficient conditions. In
fact, we were not able to prove unbounded behaviour in any model. Our
hypothesis is that if the distribution of the variable $\|\psi\|$ in
Example~\ref{ex:cld} is heavy-tailed (the shootout distance
distribution before scaling), this should produce clouds of particles
which spread indefinitely.

As already mentioned in the previous Section, existence of a limiting
measure is also an open question for all the models but two Considers
there. We tend to think that boundedness of a model should suffice for
a weak limit to exist.

\begin{Ack}
  The authors are grateful to Richard Darling who shared the initial
  model with one of us (SZ). We also thank Mikhail Menshikov and
  Victor Kleptsyn for useful and inspiring discussions. SZ also
  acknowledges hospitality of the University of Berne where a part of this
  work has been done.
\end{Ack}

\bibliographystyle{plain}
\bibliography{dirichlet}

\end{document}